\documentclass[11pt]{amsart}

\usepackage[T1]{fontenc}
\usepackage[utf8]{inputenc}
\usepackage{amsmath,amssymb,amsfonts,mathrsfs}
\usepackage{pgfplots}
\pgfplotsset{compat=1.16}
\usetikzlibrary{arrows}
\usepackage{hyperref}

\newtheorem{Theorem}{Theorem}
\newtheorem{proposition}{Proposition}
\newtheorem{corollary}{Corollary}
\newtheorem{lemma}{Lemma}
\theoremstyle{definition}
\newtheorem{definition}{Definition}
\newtheorem{notation}{Notation}
\newtheorem{note}{Note}
\newtheorem{example}{Example}
\newtheorem{algorithm}{Algorithm}

\title[A Note on Decomposable and Reducible Integer Matrices]{A Note on Decomposable and Reducible Integer Matrices}

\author{Carlos Mariju\'an}
\address{Departamento de Matem\'atica Aplicada/IMUVA, Universidad de Valladolid, Valladolid, Spain}
\email{marijuan@uva.es}

\author{Ignacio Ojeda}
\address{Departamento de Matem\'aticas/IMUEX, Universidad de Extremadura, Badajoz, Spain}
\email{ojedamc@unex.es}

\author{Alberto Vigneron-Tenorio}
\address{Departamento de Matem\'aticas/INDESS, Universidad de C\'adiz, Jerez de la Frontera, Spain}
\email{alberto.vigneron@uca.es}

\keywords{Integer matrix, Hermite normal form, decomposable matrix, reducible matrix, disconnected graph}

\begin{document}

\begin{abstract}
We propose necessary and sufficient conditions for an integer matrix to be decomposable in terms of its Hermite normal form. Specifically, to each integer matrix, we associate a symmetric integer matrix whose reducibility can be efficiently determined by elementary linear algebra techniques, and which completely determines the decomposability of the first one.
\end{abstract}

\maketitle

\section{Introduction}

For integer valued matrices, the notion of decomposability can be stated analogously to the real case (see Definition \ref{Def1}). The main difference here is that unimodularity is required for the transformation matrices. This is necessary to preserve the $\mathbb{Z}-$module structure generated by the columns of the matrix. Thus, if one wants to keep the group structure unchanged, pure Linear Algebra techniques cannot be applied to study the decomposability of an integer matrix.

Let $m \leq n$ be two positive integers. Given an $m \times n$ integer matrix $A$, we can consider the submonoid $S$ of $\mathbb{Z}^m$ generated by the non-negative combinations of the columns of $A$. A decomposition of $A$ yields a decomposition of $S$. In \cite{decomposable}, the authors deal with the computation of the decompositions of $S$, if possible, using the (integer) Hermite normal form as the main tool. Following this idea, we relate the decomposition of any integer matrix and the decomposition of its Hermite normal form (Proposition \ref{Prop1}). This leads to our main result (Theorem \ref{Th main}) which states that if $H$ is the Hermite normal form of an integer matrix $A$, the necessary and sufficient condition for $A$ to be decomposable is that a certain symmetric matrix is reducible in the usual sense (see Definition \ref{Def2}). Now,  we can adapt the combinatorial and Linear Algebra machinery to determine if $A$ is decomposable: note that, for a symmetric real matrix, it is possible to decide if it can be decomposed into a direct sum of smaller symmetric real matrices by analyzing the connectivity of a certain associated graph, which is closely related to the spectral properties of the graph. All this allows us to propose an algorithm (Algorithm \ref{Algo}) for the computation of the decomposition of the matrix $A$, if possible.

Apart from practical computational considerations, we emphasize that, given an integer matrix $A$, we propose a new approach by associating $A$ with a simple graph whose connectivity determines its decomposition. Consequently, 
this can be used to determine the decomposition of any finitely generated commutative submonoid of $\mathbb{Z}^m$,  as an alternative method to \cite{decomposable}. Recall that the study of finitely generated commutative submonoids  of $\mathbb{Z}^m$ is of great interest due to its close relation with Toric Geometry (see \cite{Cox, Miller05} or \cite{Rosales}, and the references therein). Moreover, in this context, integer decomposable matrices have their own importance; to mention a couple of illustrative examples we observe that decomposable graphical models have associated integer decomposable matrices, as can be deduced from \cite[Theorem 4.2]{Sturmfels}, and that decomposable semigroups correspond to direct products of certain algebraic (toric, in a wide sense) varieties.

\section{On decomposable and reducible integer matrices}

Let $m \leq n$ be two positive integers.

\begin{definition}\label{Def1}
Let $A \in \mathbb{Z}^{m \times n}$. We say that $A$ is \textbf{decomposable} if there exist a unimodular matrix $P$ and a permutation matrix $Q$ such that $P^{-1} A Q$ decomposes into a direct sum of matrices.
\end{definition}

As mentioned in the introduction, the main purpose of this note is to study decomposable matrices in terms of their Hermite normal form. Let us recall the notion of Hermite normal form of an integer matrix.

\begin{definition}\label{DefHNF}
Let $A \in \mathbb{Z}^{m \times n}$ of rank $r$. The \textbf{Hermite normal form} of $A, \operatorname{HNF}(A)$, is the unique matrix $H=(h_{ij}) \in \mathbb{Z}^{m \times n}$ such that $A = PH$, for a unimodular matrix $P$, satisfying the following three conditions:
\begin{itemize}
\item[(a)] there exists a sequence of integers $j_1, \ldots, j_r$ such that $1 \leq j_1 < \ldots < j_r \leq n$, and for each $1 \leq i \leq r$ we have $h_{i j_i} > 0$ and $h_{ij} = 0$ for all $j < j_i$ (row echelon form);
\item[(b)] for $1 \leq k < i \leq r$ we have $0 \leq h_{k\, j_i} < h_{i\, j_i}$ (the pivot element is the greatest along its column and the coefficients above are non-negative);
\item[(c)] the last $m-r$ rows of $H$ are zero.
\end{itemize}
We say that $A$ is in Hermite normal form when $A = \operatorname{HNF}(A)$.
\end{definition}

There are well-known efficient algorithms for the computation of the Hermite normal form of an integer matrix (see, e.g. \cite{HNF}). They are implemented in the usual computer algebra systems; for example, in GAP (\cite{gap})
and Mathematica (\cite{mathematica}), the commands \texttt{HermiteNormalFormIntegerMat} and \texttt{HermiteDecomposition}, respectively, compute the Hermite normal form of an integer matrix.

\begin{example}\label{Ejem1}
The Hermite normal form of \[A = \left(\begin{array}{ccccc} 2&-4&2&5&-6\\ 2&-2&2&5&-3\\ 0&-2&1&2&-3\end{array}\right)\] is \[\operatorname{HNF}(A) = \left(\begin{array}{rrr} 1&0&-2 \\ -1&1&0\\ -1&1&1 \end{array}\right)\, A = \left(\begin{array}{ccccc}2&0&0&1&0\\ 0&2&0&0&3 \\ 0&0&1&2&0\end{array}\right),\] where the matrix $\left( \begin{array}{rrr}
  1  &  0  & -2  \\
 -1  &  1  &  0  \\
 -1  &  1  &  1
\end{array} \right)$  is the product of the elementary matrices transforming the matrix $A$ into its Hermite normal form as above, in such a way that the unimodular matrix in Definition \ref{DefHNF} is \[P=\left( \begin{array}{rrr}
  1  &  0  & -2  \\
 -1  &  1  &  0  \\
 -1  &  1  &  1
\end{array} \right)^{-1}=\left( \begin{array}{rrr}
  1  & -2  &  2  \\
  1  & -1  &  2  \\
  0  & -1  &  1
\end{array} \right).\]
\end{example}

The next propositions provide necessary and sufficient conditions for an integer matrix to be decomposable in terms of its Hermite normal form.

\begin{proposition}\label{Prop1}
Let $A \in \mathbb{Z}^{m \times n}$ and let $H = \operatorname{HNF}(A)$. Then, $A$ is decomposable if and only if $H$ is decomposable.
\end{proposition}

\begin{proof}
Let $P_1$ be a unimodular matrix such that $P_1^{-1} A = H$. If $A$ is decomposable, then $A_1 \oplus \ldots \oplus A_t = P_2^{-1} A Q = P_2^{-1} P_1 H Q = (P_1^{-1} P_2)^{-1} H Q$, for a unimodular matrix $P_2$ and a permutation matrix $Q$. Now, since $P_1^{-1} P_2$ is unimodular, we have that $H$ is decomposable. Conversely, assume that $H$ is decomposable, so there exist a unimodular matrix $P_3$ and a permutation matrix $Q_1$ such that $P_3^{-1} H Q_1 = H_1 \oplus \ldots \oplus H_s$. Thus, $H_1 \oplus \ldots \oplus H_s = P_3^{-1} P_1^{-1} A Q_1 = (P_1 P_3)^{-1} A Q_1$ and we are done.
\end{proof}

In the following, we use the symbol $\top$ to denote the transpose operation.

\begin{proposition}\label{Prop2}
Let $H$ be an $m \times n$ integer matrix in Hermite normal form. Then, $H$ is decomposable if and only if there exist permutation matrices $P$ and $Q$ such that $P^\top H Q$ decomposes into a direct sum of matrices.
\end{proposition}

\begin{proof}
First, we observe that if the rank of $H$ is $r < m$, then the last $m-r$ rows of $H$ are zero. As these rows do not affect the condition of $H$ to be decomposable, we assume that $H$ has rank $m$.

The sufficiency part is obvious since the permutation matrix $P$ is unimodular and $P^\top = P^{-1}$. Conversely, if $H$ is decomposable, there exist a unimodular matrix $R$ and a permutation matrix $Q$ such that $R^{-1} H Q = A_1 \oplus \ldots \oplus A_t.$ For simplicity, we assume that $t=2$. Let $P_1$ and $P_2$ be unimodular matrices such that $H_1 := P_1^{-1} (A_1 \vert\, 0) Q^\top$ and $H_2 := P_2^{-1} (0\, \vert A_2) Q^\top$ are in Hermite normal form, and define the following matrix $$B := \left(\begin{array}{c} H_1 \\ \cline{1-1} H_2 \end{array}\right) = (P_1 \oplus P_2)^{-1} (A_1 \oplus A_2) Q^\top = (P_1^{-1} A_1 \oplus P_2^{-1} A_2) Q^\top.$$ Since the rank of $B$ is $m$, each row of $B$ contains a pivot element of $H_1$ or $H_2$. If we move the row containing the first (leftmost) pivot element to the first place, the row containing the second pivot element to the second place and so forth, the resulting matrix is necessarily in Hermite normal form. Thus, there exists a permutation matrix $P$ such that $P B = H$, by the uniqueness of the Hermite normal form. Therefore, $H = P B = P (P_1^{-1} A_1 \oplus P_2^{-1} A_2) Q^\top$ and we conclude that $P^\top H Q$ decomposes into $P_1^{-1} A_1 \oplus P_2^{-1} A_2$.
\end{proof}

\begin{example}\label{Ejem2}
By Proposition \ref{Prop2}, we can easily see that the matrix $A$ in Example \ref{Ejem1} is decomposable. Indeed,
 \[\left(\begin{array}{ccc} 1 & 0 & 0 \\ 0 & 0 & 1 \\ 0 & 1 & 0 \end{array}\right) \operatorname{HNF}(A) \left(\begin{array}{ccccc} 1 & 0 & 0 & 0 & 0\\ 0 & 0 & 0 & 1 & 0 \\ 0& 1 & 0 &0 &0 \\ 0& 0 & 1 &0 &0 \\ 0& 0 & 0 &0 & 1 \end{array}\right) =  \left(\begin{array}{ccc|cc}2&0&1&0&0\\ 0&1&2&0&0 \\ \cline{1-5} 0&0&0&2&3\end{array}\right).\]
\end{example}


For symmetric matrices, decomposability can be refined to the more restrictive notion of reducibility. This notion has a rich combinatorial nature, because of its relationship with graph theory, as we will see later on.

\begin{definition}\label{Def2}
A symmetric matrix $B \in \mathbb{Z}^{n \times n}$ is \textbf{reducible} if there exists a permutation matrix $Q$ such that $Q^\top B Q$ decomposes into a direct sum of square matrices. Otherwise $B$ is said to be \textbf{irreducible}.
\end{definition}

The following result gives a necessary and sufficient condition for an integer matrix to be decomposable in terms of the reducibility of a certain related symmetric matrix.
To state our result we need a piece of notation.

\begin{notation}
Let $H$ be an $m \times n$ integer matrix in Hermite normal form. With the same notation as in Definition \ref{DefHNF}, we write $\operatorname{tim}(H)$ for the $n \times n$ triangular integer matrix whose $j_i-$th row is the $i-$row of $H,\ i = 1, \ldots, r$, and zeros elsewhere.
\end{notation}

The following example illustrates the above notation.

\begin{example}
If \[H = \left(\begin{array}{rrrrr} 1&1&-2&0&0 \\ 0&2&1&0&0 \\ 0&0&0&3&2 \end{array}\right),\] then \[\operatorname{tim}(H) = \left(\begin{array}{rrrrr}   1&1&-2&0&0 \\ 0&2&1&0&0 \\ 0&0&0&0&0 \\ 0&0&0&3&2 \\ 0&0&0&0&0 \end{array}\right).\]
\end{example}

Observe that the matrix $\operatorname{tim}(H)$ is not necessarily in Hermite normal form.

\begin{lemma}\label{Lema main}
Let $H$ be an $m \times n$ integer matrix in Hermite normal form. If $H$ is decomposable, then there exists a permutation matrix $Q$ such that $Q^T \operatorname{tim}(H) Q$ decomposes into a direct sum of triangular matrices; in particular, $\operatorname{tim}(H)$ is decomposable.
\end{lemma}

\begin{proof}
By Proposition \ref{Prop2}, there exist permutation matrices $P_0$ and $Q$ such that $P_0^\top H Q$ decomposes into a direct sum of matrices. We may assume that $Q$ preserves the original column order
within each block. Clearly, adding rows and columns to $P_0$ conveniently, we may construct an $n \times n$ permutation matrix $P_1$ such that $P_1^\top \operatorname{tim}(H) Q = H'_1 \oplus \ldots \oplus H'_t$. Matrices $H'_i,\ i = 1, \ldots, t$ are not necessarily triangular. However, since $\operatorname{tim}(H)$ is triangular, there exists a permutation matrix $P_2$ such \[(P_1 P_2)^\top \operatorname{tim}(H) Q= P_2^\top (P_1^\top \operatorname{tim}(H) Q) = H_1 \oplus \ldots \oplus H_t\] where $H_i$ is triangular for every $i \in \{1, \ldots, t\}$. Now, since $(P_1 P_2)^\top \operatorname{tim}(H) Q$ and $\operatorname{tim}(H)$ are both triangular, we may choose $P_2$ so that $P_1 P_2$ and $Q$ are identical up to permutation of the zero rows of $(P_1 P_2)^\top \operatorname{tim}(H) Q$.
\end{proof}

\begin{Theorem}\label{Th main}
Let $A$ be an $m \times n$ integer matrix. If $H = \operatorname{HNF}(A)$, then $A$ is decomposable if and only if \[\operatorname{tim}(H)+\operatorname{tim}(H)^\top\] is reducible.
\end{Theorem}

\begin{proof}
By Proposition \ref{Prop1}, we may assume that $A = H$. Now, if $H$ is decomposable, by Lemma \ref{Lema main}, there exists a permutation matrix $Q$ such that $Q^\top \operatorname{tim}(H) Q = H_1 \oplus \ldots \oplus H_t$, with $H_i,\ i = 1, \ldots, t,$ triangular. Therefore,
\begin{align*}
Q^\top (\operatorname{tim}(H) + \operatorname{tim}(H)^\top) Q & = Q^\top \operatorname{tim}(H) Q + Q^\top \operatorname{tim}(H)^\top Q = \\ & = Q^\top \operatorname{tim}(H) Q + (Q^\top \operatorname{tim}(H) Q)^\top = \\
& = H_1 \oplus \ldots \oplus H_t + (H_1 \oplus \ldots \oplus H_t)^\top = \\  & =
H_1 \oplus \ldots \oplus H_t + H_1^\top \oplus \ldots \oplus H_t^\top = \\ & =
(H_1 + H_1^\top) \oplus \ldots \oplus (H_t + H_t^\top),
\end{align*}
and we conclude that $\operatorname{tim}(H) + \operatorname{tim}(H)^\top$ is reducible.

Conversely, if $\operatorname{tim}(H) + \operatorname{tim}(H)^\top$ is reducible, then there exists a permutation matrix $Q$ such that $Q^\top (\operatorname{tim}(H) + \operatorname{tim}(H)^\top) Q =  H_1 \oplus \ldots \oplus H_t$. Now, since $\operatorname{tim}(H)$ is triangular, we have that $Q^\top (\operatorname{tim}(H) - \operatorname{tim}(H)^\top) Q =
{H}'_1 \oplus \ldots \oplus {H}'_t$, with ${H}'_i$ having the same order than $H_i$, for each $i \in \{1, \ldots, t\}$ respectively. So, it follows that
\begin{align*}
Q^\top \operatorname{tim}(H) Q & = Q^\top \left(\frac{1}2 \left(\operatorname{tim}(H) + \operatorname{tim}(H)^\top \right) + \frac{1}2\left(\operatorname{tim}(H) - \operatorname{tim}(H)^\top \right) \right) Q = \\
& =  \frac{1}2 \left( (H_1 + {H}'_1) \oplus \ldots \oplus (H_t + {H}'_t) \right),
\end{align*}
and we conclude that $\operatorname{tim}(H)$ is decomposable.
\end{proof}

\begin{example}\label{Ejem3}
We already know that the matrix $A$ in Example \ref{Ejem1} is decomposable. Thus, in the light of Theorem \ref{Th main}, the symmetric matrix
$B := \operatorname{tim}(\operatorname{HNF}(A))+\operatorname{tim}(\operatorname{HNF}(A))^\top$ must be reducible. Indeed,
\[B = \left(\begin{array}{ccccc}  4&0&0&1&0\\ 0&4&0&0&3 \\ 0&0&2&2&0  \\ 1&0&2&0&0 \\ 0&3&0&0&0 \end{array}\right),\] and \[Q^\top B Q = \left(\begin{array}{ccc|cc} 4 & 0 & 1 & 0 & 0\\
0 & 2 & 2 & 0 & 0\\
1 & 2 & 0 & 0 & 0\\ \cline{1-5}
0 & 0 & 0 & 4 & 3\\
0 & 0 & 0 & 3 & 0 \end{array}\right),\quad \text{with}\quad Q = \left(\begin{array}{ccccc}  1 & 0 & 0 & 0 & 0 \\ 0 & 0 & 0 & 1 & 0\\ 0 & 1 & 0 & 0 & 0 \\ 0 & 0 & 1 & 0 & 0 \\ 0 & 0 & 0 & 0 & 1 \end{array}\right).\]
\end{example}

\section{The simple graph of a integer matrix. HNF-decomposition algorithm}

An important advantage of dealing with symmetric matrices is their strong combinatorial meaning: any symmetric matrix $B = (b_{ij}) \in \mathbb{Z}^{n \times n}$ can be considered as the adjacency matrix of an undirected graph $\mathcal{G}_B$ with $n$ vertices $\{1, \ldots, n\}$ such that $\{i,j\}$ is an edge of $\mathcal{G}_B$ if and only if $i \neq j$ and $b_{ij} \neq 0$.

Note that we are not concerned with diagonal elements and magnitudes of $B$ to construct $\mathcal{G}_B$.

\begin{example}\label{Ejem4}
The graph $\mathcal{G}_B$ corresponding to the matrix $B$ in Example \ref{Ejem3} is
\begin{center}
\begin{tikzpicture}[line cap=round,line join=round,>=triangle 45,x=1cm,y=1cm]
\clip(-.5,-1.5) rectangle (4.5,0.5);
\draw [line width=1pt, color=black] (0,0)-- (2,0);
\draw [line width=1pt, color=black] (2,0)-- (4,0);
\draw [line width=1pt, color=black] (1,-1)-- (3,-1);
\filldraw [black] (0,0) circle (3pt) node[anchor=south west] {$1$};
\filldraw [black] (2,0) circle (3pt) node[anchor=south west] {$4$};
\filldraw [black] (4,0) circle (3pt) node[anchor=south west] {$3$};
\filldraw [black] (1,-1) circle (3pt) node[anchor=south west] {$2$};
\filldraw [black] (3,-1) circle (3pt) node[anchor=south west] {$5$};
\end{tikzpicture}
\end{center}
Notice that $\mathcal{G}_B$ is not connected in this case.
\end{example}

Clearly, \emph{a symmetric matrix $B$ is reducible if and only if the graph $\mathcal{G}_B$ is not connected}. Thus, by Theorem \ref{Th main}, we can study the decomposability of an integer matrix $A$ by means of the graph $\mathcal{G}_B$ for $B = \operatorname{tim}(\operatorname{HNF}(A)) + \operatorname{tim}(\operatorname{HNF}(A))^\top$ as follows.

\begin{corollary}
Let $A \in \mathbb{Z}^{m \times n}$ and set $B = \operatorname{tim}(\operatorname{HNF}(A)) + \operatorname{tim}(\operatorname{HNF}(A))^\top.$ Then $A$ is decomposable if and only if $\mathcal{G}_B$ is not connected.
\end{corollary}

\begin{note}\label{nota H}
Given $A \in \mathbb{Z}^{m \times n}$, the graph $\mathcal{G}_B$, with $B = \operatorname{tim}(\operatorname{HNF}(A)) + \operatorname{tim}(\operatorname{HNF}(A))^\top$, can be constructed directly from $H = \operatorname{HNF}(A)$. Indeed, with the notation of Definition \ref{DefHNF}, it suffices to observe that $\{j_i,j\}\in\mathcal{G}_B$ if and only if $h_{ij}\ne0$, for $j>j_i$. Therefore, all the information concerning the decomposability of $A$ is encoded in $H$.
\end{note}

We finalize this note by giving an algorithm for the computation (if possible) of the decomposition of an $m \times n$ integer matrix into the direct sum matrices in Hermite normal form.

Let $G$ be the adjacency matrix of an undirected simple graph $\mathcal{G}$. Recall that the \textbf{degree} of the $i-$vertex of $\mathcal{G}$ is \[d_i := \sum_{\{i,j\} \in \mathcal{G}} 1,
\] and the \textbf{Laplacian matrix} of $\mathcal{G}$ is $D - G$ where $D$ is the diagonal matrix with diagonal entries $(d_1, \ldots, d_n)$.

The second part of the following result is well-known; however, for lack of a reference we sketch a proof.

\begin{proposition}\label{key proposition}
Let $\mathcal{G}$ be an undirected simple graph on $n$ vertices. Then, $\mathcal{G}$ has $t$ connected components if and only if the Laplacian matrix of $\mathcal{G}$ has rank $n-t$. In this case, the connected components of $\mathcal{G}$ are completely determined by the reduced row echelon form of the Laplacian matrix of $\mathcal{G}$.
\end{proposition}

\begin{proof}
The first statement follows from the well-known matrix-tree theorem (see, e.g. \cite[Section 1]{Merris} and the references therein). Let us analyze the second statement with a little more detail. First, we observe that the  Laplacian matrix of a connected graph on $n$ vertices is an order $n$ symmetric matrix of rank $n-1$ whose columns sum to zero. So, its reduced row echelon form is equal to \[\left(\begin{array}{rrrrr} 1 & 0 & \ldots & 0 & -1 \\ 0 & 1 & \ldots & 0 & -1 \\ \vdots & \vdots & \ddots & \vdots & \vdots \\ 0 & 0 & \ldots & 1 & -1 \\ 0 & 0 & \ldots & 0 & 0 \end{array}\right).\] Thus, if $V$ is the reduced row echelon of the Laplacian matrix of a (non-necessarily connected) undirected simple graph on $n$ vertices, then if the $j-$th column, $\mathbf{v}_j$, of $V$ is not a pivot column, the set of vertices of the connected component containing the vertex $j$ is $\{j\} \cup\, \{p_i \mid i \in \operatorname{supp}(\mathbf{v}_j)\}$, where $p_i$ is the pivot column of row $i$ and $\mathrm{supp}(\mathbf{v}_j) := \{ i \mid v_{ij} \neq 0\}$ denotes the support of $\mathbf{v}_j.$
\end{proof}

\begin{example}
Consider the graph $\mathcal{G}$ with vertex-set $\{1,2,3,4,5\}$ and edges $\{1,4\}, \{2,5\},\{3,4\}$.
The Laplacian matrix of $\mathcal{G}$ is
\[\left(\begin{array}{rrrrr} 1 & 0 & 0 & -1 & 0 \\ 0 & 1 & 0 & 0 & -1\\ 0 & 0 & 1 & -1 & 0 \\-1  & 0 & -1 & 2 & 0 \\ 0 & -1 & 0 & 0 & 1\end{array}\right)\]
and its reduced row echelon form is
\[R := \left(\begin{array}{rrrrr}
1 & 0 & 0 & -1 & 0 \\  0 & 1 & 0 & 0 & -1 \\  0 & 0 & 1 & -1 & 0 \\
   0 & 0 & 0 & 0 & 0 \\  0 & 0 & 0 & 0 & 0
\end{array}\right).\]
Now, we can read from $R$ that $\mathcal{G}$ has the following two connected components: the subgraph with vertices $\{\textbf{4},1,3\}$ and the subgraph with vertices $\{\textbf{5},2\}$.
\end{example}

We finally present the announced algorithm.

\begin{algorithm}\label{Algo} \textbf{HNF-decomposition.}\par
\noindent\textsc{Input:} An $m \times n$ integer matrix $A$.\\
\noindent\textsc{Output:} A unimodular matrix $P$ and a permutation matrix $Q$ such that $P^{-1} A Q = H_1 \oplus \ldots \oplus H_t$ with $H_i$ into Hermite normal form for every $i$.
\begin{enumerate}
\item Set $H = \operatorname{HNF}(A)$ and let $P_0$ be a unimodular matrix such that $P_0^{-1} A = H$.
\item Define the square matrix $\operatorname{tim}(H)$ and set $B = \operatorname{tim}(H) + \operatorname{tim}(H)^\top$. Let $G$ be the adjacency matrix of $\mathcal{G}_B$.
\item Let $D$ be the diagonal matrix whose elements in the main diagonal are entries of $G\, (1\ 1\ \ldots\ 1)^\top$ , and define $L = D - G$. If $n - \operatorname{rank}(L) = 1,$ then return $P = P_0$ and $Q$ equal to the identity matrix.
\item Let $R$ be the reduced row echelon form of $L$, with pivot columns $p_1<\cdots<p_s$, and let $k=0$.
\item For $j=1$ to $n$ do
\begin{enumerate}
\item[] If the $j-$th column, $\mathbf{v}_j$, of $R$ is a non-pivot column, then
\begin{enumerate}
\item Set $k = k+1$ and $\ell_k$ equal to the cardinality of $\operatorname{supp}(\mathbf{v}_j)$.
\item Let $Q_k$ be the $n \times (\ell_k+1)-$matrix whose columns are $\{\mathbf{e}_j\} \cup \{\mathbf{e}_{p_i} \mid i \in \operatorname{supp}(\mathbf{v}_j)\},$ where $\mathbf{e}_k$ denotes the $n-$dimen\-sional vector that has the $k-$th coordinate equal to $1$ and all the other coordinates equal to $0$.
\end{enumerate}
\end{enumerate}
\item Set $Q = (Q_1 \vert \ldots \vert Q_t)$.
\item Let $P_1$ be the unimodular matrix such that $P_1^{-1} (H Q) = \operatorname{HNF}(H Q)$.
\item Return $P = P_0 P_1$ and $Q$.
\end{enumerate}
\end{algorithm}

\noindent\textbf{Observations:}\\
$\bullet$ Steps (1)-(6) provide unimodular matrices $P_0$ and $Q$ such that $P_0^{-1} A Q = \tilde{P}\, (A_1 \oplus \ldots \oplus A_t)$, for some permutation matrix $\tilde{P}$.\\
$\bullet$ Clearly $t = n - \operatorname{rank}(L)$; moreover, we have that $ \operatorname{rank}(L) = \sum_{k=1}^t \ell_k$.\\
$\bullet$ If $n - \operatorname{rank}(L)=1$ then $A$ is not decomposable. In this case $A_1 = \operatorname{HNF}(A)$ and $Q$ is the identity matrix. Otherwise, if $A$ is decomposable, we cannot guarantee that $\tilde{P} = I_m$ and that the matrices $A_i,\ i = 1, \ldots, t$, are in Hermite normal form. Assign all zero rows to $A_t$. However, since $\operatorname{HNF}(A_1) \oplus \ldots \oplus \operatorname{HNF}(A_t) = \operatorname{HNF}(A_1 \oplus \ldots \oplus A_t)$, by the uniqueness of the Hermite normal form, step (7) provides the matrix $P_1$ such that $P_1^{-1} P_0^{-1} A Q$ is in Hermite normal form as desired.\\
$\bullet$ By Note \ref{nota H}, we may replace the Step (2) by
\begin{itemize}
\item[(2)] Let $G$ be the adjacency matrix of the graph $\mathcal{G}$ with vertices $\{1, \ldots, n\}$ such that $\{j_i,j\}\in\mathcal{G}$ if and only if $h_{ij}\ne0$, for $j>j_i$.
\end{itemize}
This is advantageous for small $n$.\\
$\bullet$ An HNF-decomposition, if it exists, is not unique. It depends on the choice of the order of the columns of the matrices $Q_i,\ i = 1, \ldots, t$ and the order in which these matrices are placed.

\begin{example}
The matrix \[H = \left(\begin{array}{ccccccc}
1 & 0 & 1 & 1 & 2 & 0 & 0\\
0 & 1 & 0 & 0 & 0 & 0 & 3\\
0 & 0 & 3 & 0 & 0 & 3 & 0\\
0 & 0 & 0 & 2 & 1 & 2 & 0
\end{array}\right)\] is in Hermite normal form and its associated graph $\mathcal{G}$ (see Note \ref{nota H}) is
\begin{center}
\begin{tikzpicture}[line cap=round,line join=round,>=triangle 45,x=1cm,y=1cm]
\clip(-.5,-.5) rectangle (4.5,3.5);
\draw [line width=1pt, color=black] (0,0)-- (2,0);
\draw [line width=1pt, color=black] (0,0)-- (0,2);
\draw [line width=1pt, color=black] (0,2)-- (1,3);
\draw [line width=1pt, color=black] (1,3)-- (2,2);
\draw [line width=1pt, color=black] (2,2)-- (2,0);
\draw [line width=1pt, color=black] (0,2)-- (2,2);
\draw [line width=1pt, color=black] (3,0)-- (3,3);
\filldraw [black] (0,0) circle (3pt) node[anchor=south east] {$3$};
\filldraw [black] (2,0) circle (3pt) node[anchor=south west] {$6$};
\filldraw [black] (0,2) circle (3pt) node[anchor=south east] {$1$};
\filldraw [black] (1,3) circle (3pt) node[anchor=south west] {$5$};
\filldraw [black] (2,2) circle (3pt) node[anchor=south west] {$4$};
\filldraw [black] (3,0) circle (3pt) node[anchor=south west] {$2$};
\filldraw [black] (3,3) circle (3pt) node[anchor=south west] {$7$};
\end{tikzpicture}
\end{center}
Therefore, $H$ is a decomposable matrix. Of course, we do not need to construct the graph $\mathcal{G}$ to compute an HNF-decomposition of $H$.

In order to compute effectively two matrices $P$ and $Q$ such that $P^{-1} H Q$ is in HNF-decomposed form, we compute the reduced row echelon form, $R$, of the Laplacian matrix of $\mathcal{G}_B,$ for $B = \operatorname{tim}(H) + \operatorname{tim}(H)^\top$,
\[R = \left(\begin{array}{rrrrrrr} 1 & 0 & 0 & 0 & 0 & -1 & 0 \\ 0 & 1 & 0 & 0 & 0 & 0 & -1 \\ 0 & 0 & 1 & 0 & 0 & -1 & 0 \\ 0 & 0 & 0 & 1 & 0 & -1 & 0 \\ 0 & 0 & 0 & 0 & 1 & -1 & 0 \\ 0 & 0 & 0 & 0 & 0 & 0 & 0 \\ 0 & 0 & 0 & 0 & 0 & 0 & 0 \end{array}\right).\]
Now, following Steps (5)-(7) in Algorithm \ref{Algo}, we may take the matrix $Q$ equal to $(\mathbf{e}_6| \mathbf{e}_1| \mathbf{e}_3| \mathbf{e}_4 | \mathbf{e}_5 |\mathbf{e}_7 | \mathbf{e}_2),$ in this case, the corresponding matrix $P$ is
\[\left(\begin{array}{rrrr}
0 & 1 & 0 & 0\\
0 & 0 & 0 & 1\\
3 & 0 & -1 & 0 \\
2 & 0 & -1 & 0
\end{array}\right).\]
Now, we can check that \[P^{-1} H Q = \left(\begin{array}{rrrrr|rr}
1 & 0 & 3 & -2 & -1 & 0 & 0\\
0 & 1 & 1 & 1 & 2 & 0 & 0\\
0 & 0 & 6 & -6 & -3 & 0 & 0\\
\hline
0 & 0 & 0 & 0 & 0 & 3 & 1
\end{array}\right).\]
As mentioned above, other choices of $Q$ determine a different $P$ and, consequently, another HNF-decomposition, equivalent to the one given.
\end{example}

\section{Conclusions and future work}

Using the Hermite normal form as the main tool, we have obtained a theoretical criterion to determine whether a given integer matrix decomposes into a direct sum of lower order integer matrices.

This criterion allows us to associate a simple graph to the integer matrix whose connectedness determines the decomposition of the integer matrix and facilitates the formulation of an algorithm to decompose an integer matrix into a direct sum of matrices in Hermite normal form, provided such decomposition exists.

Our results have immediate applications to the study of affine semigroups and semigroup algebras; in fact, this was our original motivation for tackling this problem. However, we believe that our results can be generalized to matrices with entries in any Euclidean ring further than $\mathbb{Z}$. 

Since we were only interested in decomposition issues, we underestimated a lot of information from $B$ when constructing the graph $\mathcal{G}_B$. Alternatively, $B$ 
can be considered as the adjacency matrix of an (undirected) weighted graph $\mathcal{G}_B$ with $n$ vertices $\{v_1, \ldots, v_n\}$, where the weight of the edge $\{v_i, v_j\}$ is $b_{ij}$. This alternative graph is sensitive to all the information recorded in the entries of $B$. One might wonder if this alternative graph can be used to provide information for integer matrices, beyond decomposability/reducibility.

\section*{Author Contributions}

All three authors contributed equally to the investigation
and approved the published version of the manuscript.

\section*{Funding}

This work was partially supported by the Ministerio de
Econom\'ia y Competitividad (Spain) and FEDER-UE through
grants PGC2018-096446-B-C21 and MTM2017-84890-P;
by the Junta de Extremadura (Spain)/FEDER,
research group FQM-024; and by the Junta de Andaluc\'ia
(Spain), research group FQM-366.

\section*{Conflicts of Interest}

The authors declare no conflict of interest.


\begin{thebibliography}{999}
\bibitem{HNF}
\textsc{Clement, P.; Stein, W.}
\emph{Fast computation of Hermite normal forms of random
integer matrices}. Journal of Number Theory, \textbf{130} (2010), 1675--1683.

\bibitem{Cox}
\textsc{Cox, D.; Little, B. and Schenck, H.} \emph{Toric varieties}, Grad. Studies Math., 124, Amer. Math. Soc. (2011)

\bibitem{decomposable}
\textsc{García-García, J.I.;Moreno-Fr\'ias M.A.; Vigneron-Tenorio, A.}
\emph{On decomposable semigroups and applications}. Journal of Symbolic Computation
\textbf{58} (2013), 103--116.

\bibitem{gap}
The GAP~Group, \emph{GAP -- Groups, Algorithms, and Programming,
Version 4.11.1};
2021,
\url{https://www.gap-system.org}.

\bibitem{mathematica}
Wolfram Research, Inc., {\emph Mathematica}, Version 12.3, Champaign, IL (2021).

\bibitem{Merris}
\textsc{Merris, R.}
\emph{Laplacian matrices of graphs: a survey}.
Linear Algebra Appl. \textbf{197, 198} (1994), 143--176.

\bibitem{Miller05}
\textsc{Miller, E.; Sturmfels, B.}
\newblock \emph{Combinatorial Commutative Algebra}.
\newblock Vol. 227 of Graduate Texts in Mathematics. Springer, New York. 2005.

\bibitem{Rosales}
\textsc{Rosales, J.C.; García-Sánchez, P.A.}
\newblock \emph{Finitely generated commutative monoids.}
\newblock Nova Science Publishers, Inc., Commack, NY, 1999.

\bibitem{Sturmfels}
\textsc{Geiger, D.; Meek, C.; Sturmfels, B.}
\newblock \emph{On the toric algebra of graphical models}.
\newblock Ann. Statist. \textbf{34}(3), (2006), 1463-1492.

\end{thebibliography}
\end{document}